\documentclass[12pt,reqno]{amsart}
\newtheorem{theorem}{Theorem}[section]
\newtheorem{lemma}[theorem]{Lemma}
\newtheorem{proposition}[theorem]{Proposition}
\newtheorem{corollary}[theorem]{Corollary}

\numberwithin{equation}{section}

\begin{document}

\baselineskip=15pt

\title[Criterion for logarithmic connections with prescribed residues]{Criterion for 
logarithmic connections with prescribed residues}

\author[I. Biswas]{Indranil Biswas}

\address{School of Mathematics, Tata Institute of Fundamental
Research, Homi Bhabha Road, Mumbai 400005, India}

\email{indranil@math.tifr.res.in}

\author[A. Dan]{Ananyo Dan}

\address{BCAM -- Basque Centre for Applied Mathematics, Alameda de Mazarredo 14,
48009 Bilbao, Spain} 

\email{adan@bcamath.org}

\author[A. Paul]{Arjun Paul} 

\address{School of Mathematics, Tata Institute of Fundamental 
Research, Homi Bhabha Road, Mumbai 400005, India} 

\email{apmath90@math.tifr.res.in}

\subjclass[2010]{53B15, 14H60}

\keywords{Logarithmic connection, residue, rigidity, logarithmic Atiyah bundle.}

\begin{abstract}
A theorem of Weil and Atiyah says that a holomorphic vector bundle $E$ 
on a compact Riemann surface $X$ admits a holomorphic connection if and only if the 
degree of every direct summand of $E$ is zero. Fix a finite subset $S$ of $X$, 
and fix an endomorphism $A(x)\, \in\, \text{End}(E_x)$ for every $x\, \in\, S$. It is natural to ask 
when there is a logarithmic connection on $E$ singular over $S$ with residue $A(x)$
at every $x\, \in\, S$. We give a necessary and sufficient condition for it under the
assumption that the residues $A(x)$ are rigid.
\end{abstract}

\maketitle

\section{Introduction}

Let $X$ be a compact connected Riemann surface. Let $E$ be a holomorphic vector bundle 
on $X$. A holomorphic connection on $E$ is a flat connection on $E$ such that the 
(locally defined) flat sections of $E$ are holomorphic. It is known that $E$ admits a 
holomorphic connection if and only if every direct summand of $E$ is of degree zero 
\cite{We}, \cite{At}. In particular, an indecomposable holomorphic vector bundle on $X$ 
admits a holomorphic connection if and only if its degree is zero.

Fix a finite subset $S\, \subset\, X$. For each point $x\, \in\, S$, fix an 
endomorphism $A(x)\, \in\, \text{End}(E_x)$ of the fiber $E_x$ of the
vector bundle $E$. Here we address the 
following question: When there is a logarithmic connection on $E$ singular over $S$ 
with residue $A(x)$ at every $x\, \in\, S$?

First take the case where $E$ is simple, meaning all global holomorphic endomorphisms of 
$E$ are scalar multiplications. We prove the following (see Lemma \ref{ls}):

\begin{lemma}\label{l0}
Let $E$ be a simple holomorphic vector bundle on $X$. Then the following two statements
are equivalent:
\begin{enumerate}
\item There is a logarithmic connection on $E$ singular over $S$
with residue $A(x)$ for every $x\,\in\, S$.

\item The collection of endomorphisms $A(x)$, $x\, \in\, S$, satisfy the condition
$$
{\rm degree}(E)+\sum_{x\in S} {\rm trace}(A(x))\,=\, 0\, .
$$
\end{enumerate}
\end{lemma}

Note that Lemma \ref{l0} holds for stable vector bundles because they are simple.

An endomorphism $A\, \in\, \text{End}(E_x)$, where $x\, \in\, X$, will be called
rigid if $A$ commutes with every global holomorphic endomorphism of $E$. In other
words, $A$ is rigid if and only if $A\circ v(x)\,=\, v(x)\circ A$ for all
$v\, \in\, H^0(X,\, \text{End}(E))$.

Now take the case where the holomorphic vector bundle $E$ is indecomposable, meaning it 
does not split into a direct sum of holomorphic vector bundles of positive ranks. We 
prove the following (see Proposition \ref{prop1}):

\begin{proposition}\label{p0}
Let $E$ be an indecomposable vector bundle on $X$. For each point $x\in\, S$, fix a rigid
endomorphism
$A(x) \, \in\, {\rm End}(E_x)$. Then the following are equivalent:
\begin{enumerate}
\item There is a logarithmic connection on $E$ singular over $S$
such that the residue is $A(x)$ for every $x\,\in\, S$.

\item The collection $A(x)$, $x\, \in\, S$, satisfy the condition
$$
{\rm degree}(E)+\sum_{x\in S} {\rm trace}(A(x))\,=\, 0\, .
$$
\end{enumerate}
\end{proposition}

Using Proposition \ref{p0}, the following is proved (see Theorem \ref{thm1}):

\begin{theorem}\label{thm0}
Let $E$ be a holomorphic vector bundle on $X$.
For each point $x\in\, S$, fix a rigid endomorphism
$A(x) \, \in\, {\rm End}(E_x)$. Then the following hold:
\begin{enumerate}
\item For every direct summand $F\, \subset\, E$,
$$A(x)(F_x)\, \subset\, F_x$$ for every $x\, \in\, S$.

\item The vector bundle $E$ admits a logarithmic
connection on $E$ singular over $S$, with residue $A(x)$ at every $x\, \in\, S$, if and
only if for every direct summand $F\, \subset\, E$,
$$
{\rm degree}(F)+\sum_{x\in S} {\rm trace}(A(x)\vert_{F_x})\,=\, 0\, .
$$
\end{enumerate}
\end{theorem}

\section{Logarithmic connections and residue}

\subsection{Logarithmic connections on $E$}

Let $X$ be a compact connected Riemann surface. Let $E$ be a holomorphic vector bundle
on $X$. The fiber of $E$ over any point $x\, \in\, X$ will be denoted by $E_x$.

Fix finitely many distinct points
$$
S\, :=\, \{x_1\, ,\cdots\, , x_d\}\, \subset\, X\, .
$$
The reduced effective divisor $x_1+\ldots + x_d$ on $X$ will also be denoted by $S$.

The holomorphic cotangent bundle $\Omega_X$ of $X$ will also be denoted by $K_X$. A
logarithmic connection on $E$ 
singular over $S$ is a first order holomorphic differential operator
$$
D\, :\, E\, \longrightarrow\, E\otimes\Omega_X(\log S)\,=\, E\otimes 
K_X\otimes{\mathcal O}_X(S)
$$
satisfying the Leibniz identity which says that
$$
D(fs) \,=\, fD(s) + s\otimes (df)\, ,
$$
where $s$ is any locally defined holomorphic section of $E$ and $f$ is any locally defined
holomorphic function on $X$ \cite{De}.

For any point $x\, \in\, S$, the fiber $(K_X\otimes{\mathcal O}_X(S))_x$ is identified 
with $\mathbb C$ using the Poincar\'e adjunction formula \cite[p.~146]{GH}. More 
explicitly, for any holomorphic coordinate $z$ around $x$ with $z(x)\,=\, 0$, the image 
of $\frac{dz}{z}$ in $(K_X\otimes{\mathcal O}_X(S))_x$ is independent of the choice of 
the coordinate function; the above identification between $(K_X\otimes{\mathcal O}_X(S))_x$
and $\mathbb C$ sends this independent image to 
$1\,\in\, \mathbb C$. For a logarithmic connection $D$ on $E$ singular over $S$, and any 
$x\,\in\, S$, consider the composition
$$
E \,\stackrel{D}{\longrightarrow}\, E\otimes K_X\otimes{\mathcal O}_X(S)\, \longrightarrow\,
E_x\otimes (K_X\otimes{\mathcal O}_X(S))_x\,=\,E_x\, .
$$
The above Leibniz identity implies that it is ${\mathcal O}_X$--linear; hence this composition
is an element of $\text{End}(E_x)\,=\, \text{End}(E)_x$. This element of
$\text{End}(E_x)$, which we will denote by $\text{Res}(D,x)$,
is called the \textit{residue} of $D$ at $x$ \cite[p.~53]{De}.

Let ${\rm Diff}^i(E,\, E)$ be the holomorphic vector bundle on $X$
whose holomorphic sections over any open subset $U\, \subset\, X$ are the holomorphic
differential operators $E\vert_U\, \longrightarrow\, E\vert_U$
of order at most $i$. So we have a short exact sequence of holomorphic vector bundles 
\begin{equation}\label{f1}
0\, \longrightarrow\, \text{End}(E)\, =\, {\rm Diff}^0(E,\, E)\, \longrightarrow\,
{\rm Diff}^1(E,\, E)\, \stackrel{\sigma}{\longrightarrow}\,\text{End}(E)\otimes TX
\, \longrightarrow\, 0\, ,
\end{equation}
where $TX$ is the holomorphic tangent bundle of $X$, and $\sigma$ is the symbol homomorphism. Now
define the subsheaf
$$
\text{At}(E)(-\log S)\,:=\, \sigma^{-1}(\text{Id}_E\otimes_{\mathbb C}TX\otimes
{\mathcal O}_X(-S)) \,\subset\, {\rm Diff}^1(E,\, E)\, .
$$
It fits in the following short exact sequence of holomorphic vector bundles on $X$ 
obtained from \eqref{f1}:
\begin{equation}\label{e4}
0\, \longrightarrow\, \text{End}(E)\, \longrightarrow\,
\text{At}(E)(-\log S)\, \stackrel{\widehat\sigma}{\longrightarrow}\,TX(-\log S) \,:=\,
TX\otimes {\mathcal O}_X(-S)\, \longrightarrow\, 0\, ,
\end{equation}
where $\widehat\sigma$ is the restriction of $\sigma$ to the subsheaf $\text{At}(E)(-\log S)$.
A logarithmic connection on $E$ singular over $S$ is a
holomorphic splitting of \eqref{e4}, meaning a
${\mathcal O}_X$--linear homomorphism
\begin{equation}\label{h}
h\, :\, TX(-\log S)\, \longrightarrow\,\text{At}(E)(-\log S)
\end{equation}
such that ${\widehat\sigma}\circ h\, =\, \text{Id}_{TX(-\log S)}$.

For any $x\, \in\, S$, the composition
$$
\text{At}(E)(-\log S)_x\, \longrightarrow\, {\rm Diff}^1(E,\, E)_x \,
\stackrel{\sigma(x)}{\longrightarrow}\,(\text{End}(E)\otimes TX)_x
$$
vanishes; the above homomorphism $\text{At}(E)(-\log S)_x\, \longrightarrow\, {\rm Diff}^1(E,\, E)_x$
is given by the inclusion of $\text{At}(E)(-\log S)$ in ${\rm Diff}^1(E,\, E)$. Therefore, from
\eqref{f1} we get a homomorphism
\begin{equation}\label{j}
j_x\, :\, \text{At}(E)(-\log S)_x\, \longrightarrow\, \text{End}(E_x)\,=\,
\text{End}(E)_x\, .
\end{equation}
For any logarithmic connection $h$ as in \eqref{h}, the element
$$
j_x(h(x)(1))\,\in\, \text{End}(E_x)
$$
is the residue of $h$ at $x$; here $1\, \in\, \mathbb C$ is considered as
an element of $TX(-\log S)_x$ using the earlier mentioned identification between
$TX(-\log S)_x$ and $\mathbb C$.

\subsection{Logarithmic connections with given residues}

\begin{lemma}\label{lem1}
For any $x\, \in\, S$, the fiber ${\rm At}(E)(-\log S)_x$ has a canonical decomposition
$$
{\rm At}(E)(-\log S)_x\,=\, {\rm End}(E)_x\oplus TX(-\log S)_x\,=\,
{\rm End}(E_x)\oplus \mathbb C\, .
$$
\end{lemma}

\begin{proof}
We recall the definition of the Atiyah bundle: 
$$
\text{At}(E)\,:=\, \sigma^{-1}(\text{Id}_E\otimes_{\mathbb C} TX) \,\subset\,
{\rm Diff}^1(E,\, E)\, ,
$$
where $\sigma$ is the symbol map in \eqref{f1} \cite{At}. It fits in the Atiyah exact sequence
\begin{equation}\label{at}
0\, \longrightarrow\, \text{End}(E)\, \longrightarrow\,
\text{At}(E)\, \stackrel{\widetilde\sigma}{\longrightarrow}\, TX\, \longrightarrow\, 0\, ,
\end{equation}
where $\widetilde\sigma$ is the restriction of $\sigma$. Let
\begin{equation}\label{at2}
0\, \longrightarrow\, \text{End}(E)\otimes {\mathcal O}_X(-S)\, \longrightarrow\,
\text{At}(E)\otimes {\mathcal O}_X(-S)\, \longrightarrow\, TX\otimes {\mathcal O}_X(-S)\,
\longrightarrow\, 0
\end{equation}
be the tensor product of the Atiyah exact sequence with the line bundle
${\mathcal O}_X(-S)$. From this exact sequence and \eqref{e4} it follows immediately that
$$
\text{At}(E)\otimes {\mathcal O}_X(-S)\, \subset\, \text{At}(E)(-\log S)\, .
$$
For any $x\, \in\, S$, from this inclusion of coherent sheaves and the exact sequences
\eqref{e4} and \eqref{at} we have the commutative diagram
$$
\begin{matrix}
0 & \longrightarrow & (\text{End}(E)\otimes {\mathcal O}_X(-S))_x
& \stackrel{c}{\longrightarrow} &\text{At}(E)\otimes {\mathcal O}_X(-S)_x & \longrightarrow &
(TX\otimes {\mathcal O}_X(-S))_x & \longrightarrow & 0\\
&& ~ \Big\downarrow a && ~ \Big\downarrow b && \Vert\\
0 & \longrightarrow & \text{End}(E)_x
& \longrightarrow &\text{At}(E)(-\log S)_x & \longrightarrow &
(TX\otimes {\mathcal O}_X(-S))_x & \longrightarrow & 0
\end{matrix}
$$
where $a$ is the zero homomorphism because $x\, \in\, S$. Now from the snake lemma
(see \cite[p.~158, Lemma~9.1]{La}) it follows that the kernel 
of the homomorphism $b$ coincides with the image of $c$. Hence the image of the fiber 
$(\text{At}(E)\otimes {\mathcal O}_X(-S))_x$ in $\text{At}(E)(-\log S)_x$ is identified 
with the quotient line $TX(-\log S)_x\,=\, {\mathbb C}$ of $\text{At}(E)\otimes {\mathcal 
O}_X(-S)_x$. It is easy to see that this image $TX(-\log S)_x\, \subset\, 
\text{At}(E)(-\log S)_x$ of $b$ coincides with the kernel of the homomorphism $j_x$ in 
\eqref{j}. On the other hand, from \eqref{e4} it follows that $\text{End}(E)_x\, 
\subset\, \text{At}(E)(-\log S)_x$. Now it is straight-forward to check that the 
resulting homomorphism
$$
\text{End}(E)_x\oplus TX(-\log S)_x\,\longrightarrow\,\text{At}(E)(-\log S)_x
$$
is an isomorphism.
\end{proof}

For each point $x\in\, S$, fix an endomorphism
$$
A(x) \, \in\, \text{End}(E_x)\, .
$$
In view of Lemma \ref{lem1}, we have the complex line
$$
\ell_x\,:=\, {\mathbb C}\cdot (A(x)\, ,1) \, \subset\, \text{End}(E_x)\oplus {\mathbb C}\,=\,
\text{At}(E)(-\log S)_x\, .
$$
Let ${\mathcal A}(E)\, \longrightarrow\, X$ be the holomorphic vector bundle that fits in the
short exact sequence of coherent sheaves
$$
0\, \longrightarrow\, {\mathcal A}(E)\, \longrightarrow\,
\text{At}(E)(-\log S)\, \longrightarrow\, \bigoplus_{x\in S}
\text{At}(E)(-\log S)_x/\ell_x\, \longrightarrow\, 0\, ,
$$
where $\ell_x$ is the line constructed above. From \eqref{e4} we have the short exact sequence
\begin{equation}\label{f2}
0\, \longrightarrow\, \text{End}(E)\otimes
{\mathcal O}_X(-S)\, \longrightarrow\,\mathcal{A}(E)\,
\stackrel{\sigma'}{\longrightarrow}\,TX(-\log S)\, \longrightarrow\, 0\, ,
\end{equation}
where $\sigma'$ is the restriction of the homomorphism $\widehat\sigma$ in \eqref{e4}.

The following lemma is a straight-forward consequence of the definitions of logarithmic
connection and residue:

\begin{lemma}\label{lem2}
A logarithmic connection on $E$ singular over $S$ with residue $A(x)$, $x\, \in\, S$, is
a holomorphic homomorphism
$$
h\, :\, TX(-\log S)\, \longrightarrow\,\mathcal{A}(E)
$$
such that $\sigma'\circ h\,=\, {\rm Id}_{TX(-\log S)}$, where $\sigma'$ is the
homomorphism in \eqref{f2}.
\end{lemma}

\begin{lemma}\label{lem3}
For each point $x\, \in\, S$, fix a number $\lambda_x\,\in\, \mathbb C$. Let $L$ be a holomorphic line bundle
on $X$. There is a logarithmic connection on $L$ singular over $S$ with residue $\lambda_x$ at
each $x\, \in\, S$ if and only if
\begin{equation}\label{g1}
{\rm degree}(L)+\sum_{x\in S}\lambda_x \,=\, 0\, .
\end{equation}
\end{lemma}

\begin{proof}
If $D$ is a logarithmic connection on $L$ of the above type, then from \cite[p.~16,
Theorem~3]{Oh} we know that \eqref{g1} holds.

To prove the converse, assume that \eqref{g1} holds. Fix a point $x_0\, \in\, X\setminus S$.
Since the two holomorphic line bundles
$L$ and $L_0\,:=\, {\mathcal O}_X({\rm degree}(L)\cdot x_0)$ differ by tensoring with a
holomorphic line bundle of degree zero, and any holomorphic line bundle of degree zero
on $X$ admits
a nonsingular holomorphic connection (this follows immediately from the
Atiyah--Weil criterion), it suffices to show that $L_0$ admits a
logarithmic connection singular over $S$ with residue $\lambda_x$ at
each $x\, \in\, S$.

For the exact sequence of coherent sheaves
$$
0\, \longrightarrow\, K_X\, \longrightarrow\, K_X\otimes {\mathcal O}_X(S+x_0)
\, \longrightarrow\, \bigoplus_{y\in S\cup \{x_0\}} {\mathbb C}_y
\, \longrightarrow\, 0\, ,
$$
where ${\mathbb C}_y$ is a copy of $\mathbb C$ supported at $y$, let
$$
H^0(X, \, K_X\otimes {\mathcal O}_X(S+x_0))\, \longrightarrow\, 
\bigoplus_{y\in S\cup \{x_0\}} {\mathbb C}_y \, \stackrel{\xi}{\longrightarrow}
\, H^1(X, \, K_X)\, =\, {\mathbb C}
$$
be the exact sequence of cohomologies. We note that the above homomorphism $\xi$ sends
$\bigoplus_{y\in S\cup \{x_0\}} c_y$ to $\sum_{y\in S\cup \{x_0\}} c_y$. Therefore, from
\eqref{g1} it follows that there is a holomorphic section
$$
\omega_0\, \in\, H^0(X,\, K_X\otimes {\mathcal O}_X(S+x_0))
$$
whose residue over each $x\, \in\, S$ is $\lambda_x$ while the residue over $x_0$ is
${\rm degree}(L)\,=\, -\sum_{x\in S}\lambda_x$. Now the logarithmic connection $d+\omega_0$
on ${\mathcal O}_X({\rm degree}(L)\cdot x_0)\,=\, L_0$, where $d$ is the usual de Rham differential,
has the required residues.
\end{proof}

\section{The extension class}

Isomorphism classes of extensions of the line bundle $TX(-\log S)$ by the vector 
bundle $\text{End}(E)\otimes {\mathcal O}_X(-S)$ are parametrized by
$$
H^1(X,\, \text{Hom}(TX(-\log S),\, \text{End}(E)\otimes
{\mathcal O}_X(-S)))\,=\, H^1(X,\, \text{End}(E)\otimes K_X)\, .
$$
Consider the extension in \eqref{f2}. Let
\begin{equation}\label{a1}
\phi^A_E\, \in\, H^1(X,\, \text{Hom}(TX(-\log S),\, \text{End}(E)\otimes
{\mathcal O}_X(-S)))\,=\, H^1(X,\, \text{End}(E)\otimes K_X)
\end{equation}
be the corresponding cohomology class. Our aim in this section is to determine $\phi^A_E$.

Take any point $y\, \in\, X$. We will construct a homomorphism
\begin{equation}\label{a2}
\gamma_y\, :\, \text{End}(E)_y\, \longrightarrow\, H^1(X,\, \text{End}(E)\otimes K_X)\, .
\end{equation}
For this, consider the short exact sequence of sheaves on $X$
$$
0\, \longrightarrow\, {\mathcal O}_X\, \longrightarrow\,
{\mathcal O}_X(y) \,{\longrightarrow}\, {\mathcal O}_X(y)_y\,=\, T_yX
\, \longrightarrow\, 0
$$
(recall that the Poincar\'e adjunction formula identifies the fiber
${\mathcal O}_X(y)_y$ with the fiber $T_yX$ of the holomorphic tangent bundle).
Tensoring this exact sequence with $\text{End}(E)\otimes K_X$, the following
short exact sequence is obtained
$$
0\, \longrightarrow\, \text{End}(E)\otimes K_X\, \longrightarrow\,
\text{End}(E)\otimes K_X\otimes{\mathcal O}_X(y) \, {\longrightarrow}\,\text{End}(E)_y
\, \longrightarrow\, 0\, .
$$
The homomorphism $\gamma_y$ in \eqref{a2} is the one occurring in the long exact
sequence of cohomologies associated to the above short exact sequence.

Consider the Atiyah exact sequence in \eqref{at}. Let
\begin{equation}\label{a3}
\phi^0_E\, \in\, H^1(X,\, \text{Hom}(TX,\, \text{End}(E)))\,=\,
H^1(X,\, \text{End}(E)\otimes K_X)
\end{equation}
be the extension class for it.

\begin{proposition}\label{prop-e}
The cohomology class $\phi^A_E$ in \eqref{a1} coincides with
$$
\phi^0_E+ \sum_{x\in S} \gamma_x(A(x))\, \in\, H^1(X,\, {\rm End}(E)\otimes K_X)\, ,
$$
where $\phi^0_E$ and $\gamma_x$ are constructed in \eqref{a3} and \eqref{a2}
respectively.
\end{proposition}

\begin{proof}
We will give an alternative description of the homomorphism $\gamma_y$ in \eqref{a2}.
First note that Serre duality says that
\begin{equation}\label{sd}
H^1(X,\, \text{End}(E)\otimes K_X)\,=\, H^0(X,\, \text{End}(E))^*\, .
\end{equation}
Let
\begin{equation}\label{sd2}
\widetilde{\gamma}_y\, :\, \text{End}(E)_y\, \longrightarrow\, 
H^0(X,\, \text{End}(E))^*
\end{equation}
be the homomorphism produced by $\gamma_y$ using
the above Serre duality. For any $\alpha\, \in\, \text{End}(E)_y$ and
$\beta\, \in\, H^0(X,\, \text{End}(E))$, it can be checked that
\begin{equation}\label{alpha}
\widetilde{\gamma}_y(\alpha)(\beta)\, =\, \text{trace}
(\alpha\circ (\beta(y)))\, .
\end{equation}

Next we recall a Dolbeault type description of $\phi^0_E$.
Let
\begin{equation}\label{0c}
\widetilde{\phi}^0_E\, \in\, H^0(X,\, \text{End}(E))^*
\end{equation}
be the element corresponding
to ${\phi}^0_E$ in \eqref{a3} by the isomorphism in \eqref{sd}.

A $C^\infty$ homomorphism $$H\, :\, TX\, \longrightarrow\, \text{At}(E)$$
such that ${\widetilde\sigma}\circ H\,=\, \text{Id}_{TX}$ (see \eqref{at}
for $\widetilde\sigma$) defines a $C^\infty$ connection on $E$ whose
$(0\, ,1)$--component coincides
with the Dolbeault operator on the holomorphic vector bundle $E$. Let ${\mathcal K}(H)$
be the curvature of the $C^\infty$ connection on $E$ given by $H$; it
is a $(1,1)$--form with values in $\text{End}(E)$. Therefore, ${\mathcal K}(H)$
represents an element of Dolbeault cohomology $H^1(X,\, \text{End}(E)\otimes K_X)$.
This element of $H^1(X,\, \text{End}(E)\otimes K_X)$ represented by ${\mathcal K}(H)$
coincides with $\phi^0_E$ in \eqref{a3}. Consequently, for any
$\beta\, \in\, H^0(X,\, \text{End}(E))$, we have
$$
\widetilde{\phi}^0_E(\beta)\,=\, \int_X \text{trace}({\mathcal K}(H)\circ \beta)\, ,
$$
where $\widetilde{\phi}^0_E$ is defined in \eqref{0c}.

The cohomology class $\phi^A_E$ has a similar description.
Using \eqref{sd}, the cohomology class $\phi^A_E$ in 
\eqref{a1} defines an element
\begin{equation}\label{ph}
\widetilde{\phi}^A_E\, \in\, H^0(X,\, \text{End}(E))^*\, .
\end{equation}
Let
$$
H'\, :\, TX(-\log S)\, \longrightarrow\,\mathcal{A}(E)
$$
be a $C^\infty$ homomorphism such that
\begin{itemize}
\item $H'$ is holomorphic around the points of $S$, and

\item $\sigma'\circ H'\, =\, \text{Id}_{TX(-\log S)}$, where $\sigma'$ is the
homomorphism in \eqref{f2}.
\end{itemize}
Let ${\mathcal K}(H')$ be the curvature of the singular $C^\infty$ connection on
$E$ defined by $H'$. We note that ${\mathcal K}(H')$ is a $C^\infty$ two-form on $X$
with values in $\text{End}(E)$; this is because ${\mathcal K}(H')$ vanishes in a
neighborhood of $S$, so it is an $\text{End}(E)$ valued $C^\infty$ two-form on
entire $X$. The Dolbeault cohomology class in $H^1(X,\, \text{End}(E)\otimes K_X)$
represented by the $\text{End}(E)$--valued $(1,\, 1)$--form
${\mathcal K}(H')$ on $X$ coincides with $\phi^A_E$ in \eqref{a1}. Consequently,
for any $\beta\, \in\, H^0(X,\, \text{End}(E))$, we have
$$
\widetilde{\phi}^A_E(\beta)\,=\, \int_X \text{trace}({\mathcal K}(H')\circ \beta)\, .
$$
The proposition is a straight-forward consequence of the above descriptions
of $\gamma_y$, $\phi^0_E$ and $\phi^A_E$.
\end{proof}

\begin{lemma}\label{ls}
Let $E$ be a simple vector bundle, meaning $H^0(X,\, {\rm End}(E))\,=\,
{\mathbb C}\cdot {\rm Id}_E$. Then the following two statements are equivalent:
\begin{enumerate}
\item There is a logarithmic connection on $E$ singular over $S$
with residue $A(x)$ for every $x\,\in\, S$.

\item The collection of endomorphisms $A(x)$, $x\, \in\, S$, satisfy the condition
$$
{\rm degree}(E)+\sum_{x\in S} {\rm trace}(A(x))\,=\, 0\, .
$$
\end{enumerate}
\end{lemma}

\begin{proof}
If $E$ admits a logarithmic connection singular over $S$
such that residue is $A(x)$ for every $x\,\in\, S$, then
${\rm degree}(E)+\sum_{x\in S} {\rm trace}(A(x))\,=\, 0$ \cite[p.~16,
Theorem~3]{Oh}.

To prove the converse, assume that
\begin{equation}\label{ae}
{\rm degree}(E)+\sum_{x\in S} {\rm trace}(A(x))\,=\, 0\, .
\end{equation}
In Lemma \ref{lem3}, set $L\,=\, \bigwedge^{r} E$, where $r\, =\, \text{rank}
(E)$, and $\lambda_x\,=\, \text{trace}(A(x))$ for every $x\, \in\, S$. Let
$$\widetilde{\phi}_L\, \in\, H^0(X,\, \text{End}(L))^*\,=\, {\mathbb C}^*$$ be the
class for $(L\, , \{\lambda_x\}_{x\in S})$ (as in \eqref{ph}). It is straight-forward
to check that
$$
\widetilde{\phi}^A_E(\text{Id}_E)\,=\, \widetilde{\phi}_L(1)
$$
(see \eqref{ph} for $\widetilde{\phi}^A_E$).
In view of \eqref{ae}, from Lemma \ref{lem3} it follows that $L$ has a logarithmic
connection with residue $\lambda_x$ at
every $x\, \in\, S$. Therefore, we have $\widetilde{\phi}_L(1)\,=\, 0$.
Hence from the above equality it follows that $\widetilde{\phi}^A_E(\text{Id}_E)\,=\,
0$. Now we conclude that $\widetilde{\phi}^A_E\,=\, 0$, because $E$ is simple.
\end{proof}

Since all stable vector bundles are simple, Lemma \ref{ls} has the following
corollary.

\begin{corollary}\label{cor1}
Let $E$ be a stable vector bundle. Then the following two statements are equivalent:
\begin{enumerate}
\item There is a logarithmic connection on $E$ singular over $S$
with residue $A(x)$ for every $x\,\in\, S$.

\item The collection of endomorphisms $A(x)$, $x\, \in\, S$, satisfy the condition
$$
{\rm degree}(E)+\sum_{x\in S} {\rm trace}(A(x))\,=\, 0\, .
$$
\end{enumerate}
\end{corollary}

\section{Rigid endomorphisms and logarithmic connections on vector bundles}

\subsection{Rigid endomorphisms}

The vector space
$$
H^0(X,\, \text{End}(E))\,=\, H^0(X,\, E\otimes E^*)
$$
is a Lie algebra for the operation $[v\, ,w]\,=\, v\circ w- w\circ v$. For any
point $x\, \in\, X$, let
$$
s_x\, :=\, H^0(X,\, \text{End}(E))\,\longrightarrow\, \text{End}(E_x)
\,=\, \text{End}(E)_x\, ,\ \ v \,\longmapsto\, v(x)\, ,
$$
be the evaluation map. Now define the Lie subalgebra
$$
I(x)\,:=\, \text{image}(s_x) \, \subset\, \text{End}(E_x)\, .
$$
An element $\alpha\, \in\, \text{End}(E_x)$ will be called \textit{rigid} if
$$
[\alpha\, , I(x)]\,=\, 0\, ,
$$
meaning $\alpha\circ v(x)\,=\, v(x)\circ\alpha$ for all $v\, \in\, H^0(X,\, \text{End}(E))$.

Note that if $H^0(X,\, \text{End}(E))\,=\, {\mathbb C}\cdot \text{Id}_E$, then all 
elements of $\text{End}(E_x)$ are rigid.

\subsection{Criterion for logarithmic connections on indecomposable bundles}

A holomorphic subbundle $F$ of a holomorphic vector bundle $V$ on $X$ is called a 
\textit{direct summand} of $V$ if there is another holomorphic subbundle $F'$ of $V$ such 
that the natural homomorphism
$$
F\oplus F'\, \longrightarrow\, V
$$
is an isomorphism. A holomorphic vector bundle $V$ is called
\textit{indecomposable} if there is no
direct summand $F$ of it such that $0\, <\, \text{rank}(F)\, <\, \text{rank}(V)$.

As before, $S$ is an effective reduced divisor on $X$, and $E$ is a holomorphic 
vector bundle on $X$.

\begin{proposition}\label{prop1}
Assume that $E$ is indecomposable. For each point $x\in\, S$, fix a rigid endomorphism
$A(x) \, \in\, {\rm End}(E_x)$. Then the following are equivalent:
\begin{enumerate}
\item There is a logarithmic connection on $E$ singular over $S$
such that the residue is $A(x)$ for every $x\,\in\, S$.

\item The collection $A(x)$, $x\, \in\, S$, satisfy the condition
\begin{equation}\label{e3}
{\rm degree}(E)+\sum_{x\in S} {\rm trace}(A(x))\,=\, 0\, .
\end{equation}
\end{enumerate}
\end{proposition}

\begin{proof}
As noted in the proof of Lemma \ref{ls}, if $E$ admits such a logarithmic
connection, then \eqref{e3} holds.

To prove the converse, assume that \eqref{e3} holds. In view of Lemma \ref{lem2}, we
need to show that the short exact sequence in \eqref{f2} splits, or in other
words, $\widetilde{\phi}^A_E$ in \eqref{ph} is the zero homomorphism.

Since $E$ is indecomposable, all elements of $H^0(X,\, \text{End}(E))$ are of the form 
$c\cdot \text{Id}_E + N$, where $c\, \in\, \mathbb C$ and $N$ is a nilpotent 
endomorphism of $E$ \cite[p.~201, Proposition~15]{At}. In the proof of Lemma \ref{ls} 
it was shown that \eqref{e3} implies that $\widetilde{\phi}^A_E(\text{Id}_E)\,=\, 0$.

Let $N\,\in\, H^0(X,\, \text{End}(E))$ be a nilpotent endomorphism. It is known
that $$\widetilde{\phi}^0_E(N)\,=\, 0$$ \cite[p.~202, Proposition~18(ii)]{At}
(see \eqref{0c} for $\widetilde{\phi}^0_E$). Therefore,
in view of Proposition \ref{prop-e}, to prove that $\widetilde{\phi}^A_E(N)\,=\, 0$,
it suffices to show that
\begin{equation}\label{sh}
\widetilde{\gamma}_x(A(x))(N) \,=\, 0
\end{equation}
for all $x\, \in\, S$, where $\widetilde{\gamma}_x$ is constructed in \eqref{sd2}.

Since $N(x)\,\in\, \text{End}(E_x)$ is nilpotent, and $A(x)$ commutes with $N(x)$, it
follows that $A(x)\circ N(x)$ is also nilpotent. Hence we have
$\text{trace}(A(x)\circ N(x))\,=\, 0$. In view of this, \eqref{sh}
follows from \eqref{alpha}.
\end{proof}

\subsection{Criterion for logarithmic connections with rigid residues}

Let
\begin{equation}\label{d}
E\,=\, \bigoplus_{i=1}^n E^i
\end{equation}
be a decomposition of $E$ into a direct sum of indecomposable vector bundles.
If
$$
E\,=\, \bigoplus_{i=1}^{n'} F^i
$$
is another such decomposition, then $n\,=\, n'$ and there is permutation $\delta$ of
$\{1 \, ,\cdots\, , n\}$ such that $E^i$ is holomorphically isomorphic to $F^{\delta(i)
}$ for all $1\, \leq\, i\, \leq\, n$ \cite[p.~315, Theorem~2]{At1}. It is known that
such a decomposition is obtained by choosing a maximal torus in the algebraic group
$\text{Aut}(E)$ consisting of the holomorphic automorphisms of $E$ \cite{At1}. Let
$$
{\mathcal T}\,=\, ({\mathbb G}_m)^n\, \subset\, \text{Aut}(E)
$$
be the maximal torus corresponding to the decomposition in \eqref{d}, where
${\mathbb G}_m\,=\, {\mathbb C}\setminus \{0\}$ is the multiplicative group. So
the action of any $(\mu_1\, ,\cdots\, , \mu_n)\,\in\, ({\mathbb G}_m)^n$ on $E$
sends any $(w_1\, ,\cdots\, , w_n)\, \in\, \bigoplus_{i=1}^n E^i$ to
$(\mu_1 w_1\, ,\cdots\, , \mu_n w_n)\, \in\, \bigoplus_{i=1}^n E^i$.

For each point $x\in\, S$, fix a rigid endomorphism
$A(x) \, \in\, \text{End}(E_x)$. Since $A(x)$ commutes with $\mathcal T$, it follows
that
$$
A(x)(E^i_x)\, \subset\, E^i_x
$$
for all $1\, \leq\, i\, \leq\, n$ and all $x\, \in\, S$. The endomorphism
$$
A(x)\vert_{E^i_x}\, :\, E^i_x\,\longrightarrow\, E^i_x
$$
will be denoted by $A^i_x$.

If $D\, :\, E\, \longrightarrow\, E\otimes K_X\otimes{\mathcal O}_X(S)$ is a logarithmic 
connection on $E$ singular over $S$ with residue $A(x)$ at every $x\, \in\, S$,
then the composition
$$
E^i\, \hookrightarrow\, E\, \stackrel{D}{\longrightarrow}\, E\otimes
K_X\otimes{\mathcal O}_X(S) \, \stackrel{q^i\otimes {\rm Id}}{\longrightarrow}\, E^i\otimes
K_X\otimes{\mathcal O}_X(S)\, ,
$$
where $q^i\, :\, E\, \longrightarrow\,E^i$ is the natural projection obtained from
\eqref{d}, is a logarithmic connection on $E^i$ singular over $S$ with residue $A^i_x$
at every $x\, \in\, S$. Conversely, if for every $1\, \leq\, i\,\leq\, n$
$$
D^i\, :\, E^i\, \longrightarrow\, E^i\otimes K_X\otimes{\mathcal O}_X(S)
$$
is a logarithmic connection on $E^i\,\subset\, \bigoplus_{i=1}^n E^i$ singular over $S$
with residue $A^i_x$ at every $x\, \in\, S$, then $\bigoplus_{i=1}^n D^i$ is a logarithmic
connection on $E$ singular over $S$ with residue $A(x)$ at every $x\, \in\, S$.

Note that every direct summand of $E$ is a part of a decomposition of $E$ into a 
direct sum of indecomposable vector bundles. Now using Proposition \ref{prop1} we 
have the following:

\begin{theorem}\label{thm1}
For each point $x\in\, S$, fix a rigid endomorphism
$A(x) \, \in\, {\rm End}(E_x)$. Then the following hold:
\begin{enumerate}
\item For every direct summand $F\, \subset\, E$,
$$A(x)(F_x)\, \subset\, F_x$$ for every $x\, \in\, S$.

\item The vector bundle $E$ admits a logarithmic
connection on $E$ singular over $S$, with residue $A(x)$ at every $x\, \in\, S$, if and
only if for every direct summand $F\, \subset\, E$,
$$
{\rm degree}(F)+\sum_{x\in S} {\rm trace}(A(x)\vert_{F_x})\,=\, 0\, .
$$
\end{enumerate}
\end{theorem}

\section*{Acknowledgements}

We thank the referee for helpful comments. The first-named author is supported by a
J. C. Bose Fellowship.

%%%%%%%%%%%%%%%%%%%%%%%%%%%%%%%%%%%%%%%%%%%%%%%%%%%%%%%%%%%%%


\begin{thebibliography}{ZZZZ}

\bibitem[At1]{At1} M. F. Atiyah, On the Krull--Schmidt theorem with application to 
sheaves, {\it Bull. Soc. Math. Fr.} {\bf 84} (1956), 307--317.

\bibitem[At2]{At} M. F. Atiyah, Complex analytic connections in fibre
bundles, \textit{Trans. Amer. Math. Soc.} \textbf{85} (1957), 181--207.

\bibitem[De]{De} P. Deligne, \textit{Equations diff\'erentielles \`a points 
singuliers r\'eguliers}, Lecture Notes in Mathematics, Vol. 163, Springer-Verlag, 
Berlin-New York, 1970.

\bibitem[GH]{GH} P. Griffiths and J. Harris, {\it Principles of algebraic geometry}, 
Pure and Applied Mathematics. Wiley-Interscience, New York, 1978.

\bibitem[La]{La} S. Lang, \textit{Algebra}, Revised third edition, Graduate Texts in 
Mathematics, 211, Springer-Verlag, New York, 2002.

\bibitem[Oh]{Oh} M. Ohtsuki, A residue formula for Chern classes associated with 
logarithmic connections, \textit{Tokyo Jour. Math.} \textbf{5} (1982), 13--21.

\bibitem[We]{We} A. Weil, G\'en\'eralisation des fonctions ab\'eliennes, \textit{Jour. 
Math. Pure Appl.} \textbf{17} (1938), 47--87.

\end{thebibliography}
\end{document}